\documentclass[twoside,english]{amsart}
\usepackage[T1]{fontenc}
\usepackage[latin9]{inputenc}
\usepackage[a4paper]{geometry}
\usepackage{babel}
\usepackage{amstext}
\usepackage{amsthm}
\usepackage{amssymb}
\usepackage{setspace}
\usepackage[colorlinks=true,linkcolor=black,anchorcolor=black,citecolor=black,filecolor=black,menucolor=black,runcolor=black,urlcolor=black]{hyperref}

\makeatletter
\numberwithin{equation}{section}
\numberwithin{figure}{section}
\theoremstyle{plain}
\newtheorem{thm}{\protect\theoremname}
\theoremstyle{definition}
\newtheorem{defn}[thm]{\protect\definitionname}
\theoremstyle{plain}
\newtheorem{cor}[thm]{\protect\corollaryname}
\theoremstyle{remark}
\newtheorem{rem}[thm]{\protect\remarkname}
\theoremstyle{definition}
\newtheorem{example}[thm]{\protect\examplename}

\makeatother

\providecommand{\corollaryname}{Corollary}
\providecommand{\definitionname}{Definition}
\providecommand{\examplename}{Example}
\providecommand{\remarkname}{Remark}
\providecommand{\theoremname}{Theorem}

\begin{document}
\title[Note on the $q$-Log Sobolev and $p$-Talagrand Inequalities on
Carnot Groups]{Note on the $q$-Logarithmic Sobolev and $p$-Talagrand Inequalities
on Carnot Groups}
\author{E. Bou Dagher}
\address{Esther Bou Dagher: \newline Department of Mathematics \newline Imperial College London \newline 180 Queen's Gate, London SW7 2AZ \newline United Kingdom}
\email{esther.bou-dagher17@imperial.ac.uk}

\providecommand{\keywords}[1]{\textbf{\textit{Keywords---}} #1}

\begin{abstract}
In the setting of Carnot groups, we prove the $q-$Logarithmic Sobolev inequality for probability measures as a function of the Carnot-Carathéodory distance. As an application, we use the Hamilton-Jacobi equation in the setting of Carnot groups to prove the $p-$Talagrand inequality and hypercontractivity. 

\tableofcontents{}
\end{abstract}


\maketitle

\section{Introduction}

In \cite{key-57}, L. Gross obtained the Logarithmic Sobolev inequality

\begin{onehalfspace}
\begin{equation}
\int_{\mathbb{R}^{n}}f^{2}log\left(\frac{f^{2}}{\int_{\mathbb{R}^{n}}f^{2}d\mu}\right)d\mu\leq2\int_{\mathbb{R}^{n}}\vert\triangledown f\vert^{2}d\mu,\label{eq:1}
\end{equation}
where ${\displaystyle \triangledown}$ is the standard gradient on
$\mathbb{R}^{n}$ and ${\displaystyle d\mu=\frac{e^{-\frac{\vert x\vert^{2}}{2}}}{Z}d\lambda}$
is the Gaussian measure. He proved that if $\mathcal{L}$ is the non-positive
self-adjoint operator on $L^{2}\left(\mu\right)$ such that ${\displaystyle \left(-\mathcal{L}f,f\right)_{L^{2}\left(\mu\right)}=\int_{\mathbb{R}^{n}}\vert\triangledown f\vert^{2}d\mu,}$
then (\ref{eq:1}) is equivalent to the fact that the semigroup ${\displaystyle P_{t}=e^{t\mathcal{L}}}$
generated by $\mathcal{L}$ is hypercontractive: i.e. for $q\left(t\right)\leq1+\left(q-1\right)e^{2t}$
with $q>1$, we have ${\displaystyle \parallel P_{t}f\parallel_{q\left(t\right)}\leq\parallel f\parallel_{q}}$
for all $f\in L^{q}\left(\mu\right).$ 

In \cite{key-52}, D. Bakry and M. Emery extended the Logarithmic
Sobolev inequality for a larger class of probability measures defined
on manifolds under an important Curvature-Dimension condition. More
generally, if $\left(\Omega,F,\mu\right)$ is a probability space,
and $\mathcal{L}$ is a non-positive self-adjoint operator acting
on $L^{2}\left(\mu\right),$ we say that the measure $\mu$ satisfies
a Logarithmic Sobolev inequality if there is a constant $c$ such
that, for $f\in D\left(\mathcal{L}\right),$ 

\[
\int f^{2}log\frac{f^{2}}{\int f^{2}d\mu}d\mu\leq c\int f\left(-\mathcal{L}f\right)d\mu.
\]
F. Otto and C. Villani showed in their celebrated paper \cite{key-65}
that in the setting of manifolds under D. Bakry and M. Emery's Curvature-Dimension
condition, the Logarithmic Sobolev inequality implies the Talagrand
transportation cost inequality. Their proof relied on PDE's methods
and Otto calclulus in Wasserstein space \cite{key-64}. The Talagrand
transportation cost inequality was first introduced in \cite{key-66}
by M. Talagrand:
\begin{equation}
T_{w}(\mu,\nu)\leq2\int log(f)d\mu,\label{eq:2}
\end{equation}
where $\mu$ is a measure on $\mathbb{R}^{N}$ absolutely continuous
with respect to the Gaussian measure $\nu,$ $f={\displaystyle \frac{d\mu}{d\nu}}$
is the relative density, and $w(x,y)={\displaystyle \sum_{i=1}^{N}(x_{i}-y_{i})^{2}}$
is the cost of moving a unit mass from $x$ to $y.$ $T_{w}(\mu,\nu)$
is the transportation cost measuring how much effort is required to
transport a mass distributed according to $\mu$ to a mass distributed
according to $\nu$, i.e.
\[
T_{w}(\mu,\nu)={\displaystyle \underset{\pi\in\Pi(\mu,\nu)}{inf}\int_{\mathbb{R}^{N}\times\mathbb{R}^{N}}}w(x,y)d\pi(x,y),
\]
where $\Pi(\mu,\nu)$ is the set of probability measures on $\mathbb{R}^{N}\times\mathbb{R}^{N}$
with $\mu$ the first marginal and $\nu$ the second marginal. M.
Talagrand showed how the Talagrand transportation cost inequality
(\ref{eq:2}) implies the concentration of measure phenomenon, and
his approach was motivated by the work of K. Marton \cite{key-62,key-63}.
Later on, many works \cite{key-51,key-53,key-54,key-55,key-56,key-60,key-61,key-65}
studied links between the Logarithmic Sobolev inequality, the Talagrand
inequality, and the concentration of measure phenomenon, which is
one of the main tools in probability theory, statistical mechanics,
quantum information theory, stochastic dynamics, etc.
\end{onehalfspace}

In this note, in the setting of Carnot groups, we are first interested
in proving the $q-$Logarithmic Sobolev inequality (see \cite{key-72,key-71})
for $q\in\left(1,\infty\right)$, 
\begin{equation}
\int f^{q}log\frac{f^{q}}{\int f^{q}d\mu}d\mu\leq c\int\vert\triangledown_{\mathbb{}}f\vert^{q}d\mu,\label{eq:3}
\end{equation}
for measures of the form ${\displaystyle d\mu=\frac{e^{-U\left(d\right)}}{Z}d\lambda}$
. $U(d)$ is a function having a suitable growth at infinity, $\lambda$
is a natural measure like the Lebesgue measure for instance, and $d$
is a metric related to the sub-gradient ${\displaystyle \triangledown=\left(X_{1},\ldots,X_{n}\right),}$
where $X_{i}'s$ are the generators of the group's lie algebra. Since
the Laplacian is of Hörmander type and has some degeneracy, D. Bakry
and M. Emery\textquoteright s Curvature-Dimension condition in \cite{key-52}
will no longer hold true. In \cite{key-58}, W. Hebisch and B. Zegarli\'{n}ski
developed a method of studying coercive inequalities on general metric
spaces that does not require a bound on the curvature of space. Working
on a general metric space equipped with non-commuting vector fields
$\{X_{1},\ldots,X_{n}\},$ their method is based on U-bounds, where
U stands for the potential, which are inequalities of the form:
\[
\int f^{q}\eta\left(d\right)d\mu\leq C\int\vert\triangledown f\vert^{q}d\mu+D\int f^{q}d\mu,
\]
where $\eta(d)$ is a function having a suitable growth at infinity.
Those methods were applied, for example, in the works \cite{key-21-1,key-28-1,key-28-2}
in the setting of Carnot groups. In \cite{key-58}, W. Hebisch and
B. Zegarli\'{n}ski proved the $q-$Logarithmic Sobolev inequality
for the measure ${\displaystyle d\mu=\frac{e^{-\alpha d^{p}}}{Z}d\lambda}$,
where $p$ is the finite index conjugate of $q$ and $\alpha>0.$
In this note, we enlarge the class of measures that satisfy the $q-$logarithmic
Sobolev inequality to include measures of the form ${\displaystyle d\mu=\frac{e^{-U(d)}}{Z}d\lambda},$
under some growth conditions for $U(d)$, and get examples such as
$U(d)=(d+1)^{p}log(d+1)$ and $U(d)=sinh(d).$ We would like to apply
the $q-$Logarithmic Sobolev inequality to get hypercontractivity
and to obtain the $p-$Talagrand inequality on $(X,d,\mu)$ with a
constant $K:$
\begin{equation}
W_{p}(\mu,\nu)^{p}\leq\frac{1}{K}Ent_{\mu}\left(\frac{d\nu}{d\mu}\right),\label{eq:4}
\end{equation}
with $p$ finite index conjugate of $q.$ The $p-$Wasserstein distance
between two probability measures on $X$ is defined as ${\displaystyle W_{p}(\mu,\nu)^{p}={\displaystyle \underset{\pi\in\Pi(\mu,\nu)}{inf}\int_{X\times X}}d(x,y)^{p}d\pi(x,y),}$
where $\Pi(\mu,\nu)$ is the set of probability measures on $X\times X$
with $\mu$ the first marginal and $\nu$ the second marginal. \[Ent_{\mu}{\displaystyle \left(\frac{d\nu}{d\mu}\right)=\int\frac{d\nu}{d\mu}log\left(\frac{d\nu}{d\mu}\right)d\mu}\]
is the entropy functional such that $\nu$ is a probability measure
absolutely continuous with respect to $\mu.$ We note that for $p=2,$
(\ref{eq:2}) is a special case of (\ref{eq:4}).

The $p-$Talagrand inequality was introduced by Z.M. Balogh et al.
in \cite{key-67} who proved that it was implied by the $q-$Logarithmic
Sobolev inequality for $p$ the finite index conjugate of $q$; their
work generalises the result by J. Lott and C. Villani \cite{key-68}
who proved the implication for the quadratic case $p=q=2$. J. Lott
and C. Villani \cite{key-68} used the Hamilton-Jacobi infimum convolution
operator under the assumption where the space $(X,d,\mu)$ supports
local Poincaré inequality and the measure $\mu$ is a doubling measure
i.e. the measure of any open ball is positive and finite and there
exists a constant $c_{d}\geq1$ such that for all $x\in X$ and $r>0,$
\begin{equation}
{\displaystyle \mu(B(x,2r))\leq c_{d}\mu(B(x,r))}.\label{eq:5}
\end{equation}
This idea of using the Hamilton-Jacobi infimum convolution operator
was first used by S. Bobkov et al. in \cite{key-53} who explored
the equivalence between the Logarithmic Sobolev inequality with a
constant $\rho$ in $\mathbb{R}^{n}$ and hypercontractivity of the
quadratic Hamilton-Jacobi semigroup $Q_{t}$ i.e.
\[
{\displaystyle ||e^{Q_{t}f}||_{a+\rho t}\leq}||e^{f}||_{a},
\]
for every bounded measureable function $f$ on $\mathbb{R}^{n},$
$t\geq0,$ $a\in\mathbb{R},$ and $||.||_{p}$ the $L^{p}-$norm with
respect to $\mu.$

In the setting of the Heisenberg group $\mathbb{H}$, Z.M. Balogh
et al. in \cite{key-67} used W. Hebisch and B. Zegarli\'{n}ski's
\cite{key-58} $q-$Logarithmic Sobolev inequality for the measure
${\displaystyle d\mu=\frac{e^{-\alpha d^{p}}}{Z}d\lambda},$ where
$d$ is the Carnot-Carathéodory distance, to obtain the $p-$Talagrand
inequality. To prove the above implication, Z.M. Balogh et al. observed
that by Pansu's differentiability theorem \cite{key-69}, for Lipschitz
continuous functions $f:\mathbb{H}\rightarrow\mathbb{R},$ the norm
of the sub-Riemannian gradient $|\triangledown f(x)|=\left({\displaystyle \sum_{i=1}^{N}|X_{i}f|^{2}}\right)^{\frac{1}{2}}$used
in \cite{key-58} coincides with the metric subgradient in \cite{key-67}
for $\mu_{p}$ almost every $x$ for which $|\triangledown f(x)|>0.$

In section 2, we introduce the Carnot group and the Hamilton-Jacobi
equation in that setting (\cite{key-70,key-34-1}). In section 3,
we extend W. Hebisch and B. Zegarli\'{n}ski's results in \cite{key-58}
by enlarging the class of measures that satisfy the $q-$logarithmic
Sobolev inequality to include measures of the form ${\displaystyle d\mu=\frac{e^{-U(d)}}{Z}d\lambda},$
under some growth conditions for $U(d)$, and get examples such as
$U(d)=(d+1)^{p}log(d+1)$ and $U(d)=sinh(d).$ As an application,
we adapt Z.M. Balogh et al.'s Theorem 4.1 \cite{key-67} to get the
$p-$Talagrand inequality (section 4), and we adapt S. Bobkov et al.'s
Theorem 2.1 in \cite{key-53} to get hypercontractivity (section 5).
However, instead, we use the Hamilton-Jacobi equation in the setting
of Carnot groups by F. Dragoni \cite{key-70}. The advantage of doing
so is that the restriction (\ref{eq:5}) to have $\mu$ a doubling
measure is no longer required! 

\section{The Carnot Group and the Hamilton-Jacobi Equation}

Carnot groups are geodesic metric spaces that appear in many mathematical
contexts like harmonic analysis in the study of hypoelliptic differential
operators (\cite{key-74,key-75}) and in geometric measure theory
(see extensive reference list in the survey paper \cite{key-76}).
The following series of definitions are from \cite{key-73}:
\begin{defn}
We say that a Lie group on $\mathbb{R}^{N}$ , $\mathbb{G=}(\mathbb{R}^{N},\circ)$
is a (homogeneous) Carnot group if the following properties hold:

(C.1) $\mathbb{R}^{N}$ can be split as $\mathbb{R}^{N}=\mathbb{R}^{N_{1}}\times...\times\mathbb{R}^{N_{r}},$
and the dilation $\delta_{\lambda}:\mathbb{R}^{N}\rightarrow\mathbb{R}^{N}$
\[
\delta_{\lambda}(x)=\delta_{\lambda}(x^{(1)},...,x^{(r)})=(\lambda x^{(1)},\lambda^{2}x^{(2)},...,\lambda^{r}x^{(r)}),\;\;\;\;\;\;\;\;\;x^{(i)}\in\mathbb{R}^{N_{i}},
\]
 is an automorphism of the group $\mathbb{G}$ for every $\lambda>0.$
Then $\mathbb{}(\mathbb{R}^{N},\circ,\delta_{\lambda})$ is a homogeneous
Lie group on $\mathbb{R}^{N}$ . Moreover, the following condition
holds: 

(C.2) If $N_{1}$ is as above, let $Z_{1},...,Z_{N_{1}}$ be the left
invariant vector fields on $\mathbb{G}$ such that $Z_{j}(0)=\partial/\partial x_{j}|_{0}$
for $j=1,...,N_{1}$. Then
\[
rank(Lie{\{Z_{1},...,Z_{N_{1}}\}}(x))=N\;\;\;\;\;\;\forall x\in\mathbb{R}^{N}.
\]

If (C.1) and (C.2) are satisfied, we shall say that the triple $\mathbb{G}=(\mathbb{R}^{N},\circ,\delta_{\lambda})$
is a homogeneous Carnot group. We also say that $\mathbb{G}$ has
step $r$ and $N_{1}$ generators. The vector fields $Z_{1},...,Z_{N_{1}}$
will be called the (Jacobian) generators of $\mathbb{G}$.
\end{defn}

\begin{defn}
The vector valued operator $\triangledown:=(Z_{1},Z_{2},...,Z_{N_{1}})$
is called the sub-gradient on $\mathbb{G},$ and $\triangle={\displaystyle \sum_{i=1}^{N_{1}}Z_{i}^{2}}$
is called the sub-Laplacian on $\mathbb{G}.$
\end{defn}

\begin{defn}
We say that $\gamma$ is horizontal if there exist measurable functions
$a_{1},\ldots,a_{N_{1}}:\left[0,1\right]\rightarrow\mathbb{R}$ such
that
\[
\gamma'\left(t\right)=\sum_{i=1}^{N_{1}}a_{i}\left(t\right)Z_{i}\left(\gamma\left(t\right)\right)
\]
for almost all $t\in\left[0,1\right]$ i.e. $~\gamma^{'}\left(t\right)\in Span\left\{ Z_{1}\left(\gamma\left(t\right)\right),\ldots,Z_{N_{1}}\left(\gamma\left(t\right)\right)\right\} $
almost everywhere. For such a horizontal curve $\gamma,$ we define
the length of $\gamma$ to be
\[
\vert\gamma\vert=\int_{0}^{1}\left(\sum_{i=1}^{N_{1}}a_{i}^{2}\left(t\right)\right)^{\frac{1}{2}}dt.
\]
\end{defn}

\begin{defn}
The Carnot-Carathéodory distance or the control distance between two
points $x$ and $y$ is defined by
\[
d\left(x,y\right)=inf\left\{ t\vert\gamma:\left[0,t\right]\rightarrow \mathbb{G},~\gamma\left(0\right)=x,~\gamma\left(t\right)=y~\vert\gamma'\left(s\right)\vert\leq1\;\;\forall s\in\left[0,t\right]\right\} ,
\]
where $\gamma:\left[0,1\right]\rightarrow \mathbb{G}$ is an absolutely continuous
horizontal path on $\left[0,1\right],$ i.e. if for every $\ensuremath{\varepsilon>0}$
there exists a $\delta>0$ such that when a finite number of pairwise
disjoint subintervals $\left[x_{k},y_{k}\right]$ of $\left[0,1\right]$
satisfy $\ensuremath{\sum_{k}^ {}\vert y_{k}-x_{k}\vert<\delta,}$
then $\sum_{k}^ {}d\left(\gamma\left(y_{k}\right),\gamma\left(x_{k}\right)\right)<\varepsilon.$ 

It can be shown, proof found in \cite{key-77}, that $d$ is associated
to the sub-gradient through:
\[
\vert\triangledown f\left(x\right)\vert=\underset{d\left(x,y\right)\rightarrow0}{limsup}\frac{\vert f\left(x\right)-f\left(y\right)\vert}{d\left(x,y\right)}.
\]
We will be using F. Dragoni's Theorem 4 of \cite{key-70} to get hypercontractivity
and $p-$Talagrand inequality from the $q-$Logarithmic sobolev inequality.
\end{defn}

\begin{thm}[F. Dragoni \cite{key-70}]
Let $\phi:[0,\infty)\rightarrow[0,\infty)$ be a continous, non-decreasing,
and convex function with $\phi(0)=0.$ Let $f:\mathbb{R}^{n}\rightarrow\mathbb{R}$
a lower semicontinuous function such that there exists $C>0:$
\[
f(x)\geq-C(1+d(x)).
\]
Then a viscosity solution for the Hamilton-Jacobi problem in the Carnot
Group $\mathbb{G}$

$\begin{cases}
u_{t}(x,t)+\phi(|\mathbb{\triangledown}u(x,t)|)=0\text{ } & (x,t)\in\mathbb{G}\times(0,\infty)\\
u(x,0)=f(x) & x\in\mathbb{G}
\end{cases}$ \\
is given by the Hopf-Lax formula
\[
Q_{t}f(x):=\underset{y\in\mathbb{G}\mathbb{}}{inf}\left\{ t\phi^{*}\left(\frac{d(x,y)}{t}\right)+f(y)\right\} ,
\]

where $d(x,y)$ is the Carnot-Carathéodory distance as in Definition
4, $\triangledown$ is the sub-gradient as in Definition 2, and $\phi^{*}$
is the Legendre transform of $\phi.$
\end{thm}

In this note, we will be considering $f(x)$ to be a continuous and
bounded function and we will be choosing the function ${\displaystyle \phi(s)=\frac{s^{q}}{q},}$
where $1<q\leq2.$

Thus, the viscosity solution for the Hamilton-Jacobi problem in the
Carnot group $\mathbb{G}$
\begin{equation}
\begin{cases}
u_{t}(x,t)+\frac{|\triangledown u(x,t)|^{q}}{q}=0\text{ } & (x,t)\in\mathbb{}\mathbb{G}\times(0,\infty)\\
u(x,0)=f(x) & x\in\mathbb{}\mathbb{G}
\end{cases}\label{eq:}
\end{equation}

is given by the Hopf-Lax formula

\[
Q_{t}f(x)=\underset{y\in\mathbb{G}}{inf}\left\{ \frac{d(x,y)^{p}}{t^{p-1}}+f(y)\right\} ,
\]

where ${\displaystyle \frac{1}{p}+\frac{1}{q}=1.}$

\section{$q-$Logarithmic Sobolev Inequality}

Given the probability measure ${\displaystyle d\mu=\frac{e^{-U}}{Z}d\lambda},$
where $U$ is an increasing unbounded function and $Z$ is the normalization
constant. To obtain the $q-$Logarithmic Sobolev inequality, we will
need the following theorems stated in the beginning of this section
by W. Hebisch and B. Zegarli\'{n}ski in \cite{key-58}.

Let $\lambda$ be a measure satisfying the q-Poincaré inequality for
every ball ${\displaystyle B_{R}=\{x:N\left(x\right)<R\},~}$ i.e.
there exists a constant $C_{R}\in\left(0,\infty\right)$ such that
\[
\frac{1}{\vert B_{R}\vert}\int_{B_{R}}\left|f-\frac{1}{\vert B_{R}\vert}\int_{B_{R}}f\right|^{q}d\lambda\leq C_{R}\frac{1}{\vert B_{R}\vert}\int_{B_{R}}\vert\triangledown f\vert^{q}d\lambda,
\]

where $1\leq q<\infty.$

Note that we have this Poincaré inequality on balls in the setting
of Nilpotent lie groups thanks to J. Jerison's celebrated paper \cite{key-59}. 
\begin{thm}[W. Hebisch, B. Zegarli\'{n}ski \cite{key-58}]
 Let $\mu$ be a probability measure on $\mathbb{R}^{n}$ which is
absolutely continuous with respect to the measure $\lambda$ and such
that
\begin{equation}
\int f^{q}\eta d\mu\leq C\int\vert\triangledown f\vert^{q}d\mu+D\int f^{q}d\mu\label{eq:u1}
\end{equation}
with some non-negative function $\eta$ and some constants $C,D\in\left(0,\infty\right)$
independent of a function $f.$ If for any $L\in\left(0,\infty\right)$
there is a constant $A_{L}$ such that ${\displaystyle \frac{1}{A_{L}}\leq\frac{d\mu}{d\lambda}\leq A_{L}}$
on the set $\left\{ \eta<L\right\} $ and, for some $R\in\left(0,\infty\right)$
(depending on L), we have $\left\{ \eta<L\right\} \subset B_{R},$
then $\mu$ satisfies the q-Poincaré inequality
\[
\mu\vert f-\mu f\vert^{q}\leq c\mu\vert\triangledown f\vert^{q}
\]
with some constant $c\in\left(0,\infty\right)$ independent of $f.$
\end{thm}

\begin{thm}[W. Hebisch, B. Zegarli\'{n}ski \cite{key-58}]
Suppose the following Sobolev inequality is satisfied
\[
\left(\int\vert f\vert^{q+\varepsilon}d\lambda\right)^{\frac{q}{q+\varepsilon}}\leq a\int\vert\triangledown f\vert^{q}d\lambda+b\int\vert f\vert^{q}d\lambda,
\]
and the following bound is true
\begin{equation}
\mu\left(\vert f\vert^{q}\left[\vert\triangledown U\vert^{q}+U\right]\right)\leq C'\mu\vert\triangledown f\vert^{q}+D'\mu\vert f\vert^{q}.\label{eq:u2}
\end{equation}
Then, the following inequality is true
\[
\mu\left(f^{q}log\frac{f^{q}}{\mu f^{q}}\right)\leq C\mu\vert\triangledown f\vert^{q}+D\mu\vert f\vert^{q}.
\]
Moreover, if $q\in\left(1,2\right]$ and the following $q-$Poincaré
inequality holds $\mu\vert f-\mu f\vert^{q}\leq\frac{1}{M}\mu\vert\triangledown f\vert^{q},$
then one has
\begin{equation}
\mu\left(f^{q}log\frac{f^{q}}{\mu f^{q}}\right)\leq c\mu\vert\triangledown f\vert^{q}\label{eq:lsq}
\end{equation}
with some constant $c\in\left(0,\infty\right)$ independent of $f.$
\end{thm}

Under two assumptions on $d,$ outside the unit ball $B\equiv\left\{ d\left(x\right)<1\right\} ,$ 

\[
\Delta d\leq K+\beta\varepsilon d^{p-1},
\]
and
\[
\frac{1}{\sigma^{2}}\leq\vert\triangledown d\vert\leq1,
\]
where ${\displaystyle \varepsilon\in\left(0,\frac{1}{\sigma}\right),\:\beta\in\left(0,\infty\right),and\:\sigma\in\left[1,\infty\right),}$
W. Hebisch and B. Zegarli\'{n}ski proved that the measure
\[
d\mu_{p}=\frac{e^{-d^{p}}}{Z}d\lambda
\]
where $p>1$ satisfies both Poincaré and Logarithmic Sobolev inequalities.
To get the Poincaré inequality from Theorem 6, W. Hebisch and B. Zegarli\'{n}ski
proved a U-bound inequality (\ref{eq:u1}) with ${\displaystyle \eta=d^{p-1}}$
(Theorem 2.1 of \cite{key-58}). However, the Logarithmic Sobolev
inequality requires an additional U-bound inequality with ${\displaystyle U=d^{p}}$
(Theorem 2.4 of \cite{key-58}). The additional U-bound inequality
is obtained from a more general inequality of the form:
\[
\int d^{q\left(p-1\right)}f^{q}d\mu_{p}\leq C_{q}\int\vert\triangledown f\vert^{q}d\mu_{p}+D_{q}\int\vert f\vert^{q}d\mu_{p}
\]
which for a given $p\in\left(1,\infty\right)$ is valid for all $q\in\left[1,\infty\right)$
(Theorem 2.3 of \cite{key-58}). In their proof of the Logarithmic
Sobolev inequality, since ${\displaystyle U=d^{p},}$ and ${\displaystyle \frac{1}{p}+\frac{1}{q}=1,}$
then
\[
\vert\triangledown U\vert=\vert pd^{p-1}\triangledown d\vert\leq pd^{p-1}\vert\triangledown d\vert\leq pd^{p-1},
\]
which implies that
\[
\vert\triangledown U\vert^{q}+U\leq p^{q}d^{\left(p-1\right)q}+d^{p}=p^{q}d^{p}+d^{p}.
\]

Therefore, the second U-bound allows condition (\ref{eq:u2}) to be
satisfied in Theorem 7. So, they obtain the Logarithmic Sobolev inequality
for $p$ the finite conjugate of $q$. 

In the following, we present a generalised version of Theorems 2.1,
2.3, and 2.4 of \cite{key-58}: In place of the function $U(d)=d^{p},$
we let $U:\left[0,\infty\right)\rightarrow\left[0,\infty\right)$
be a twice differentiable and increasing function; and let ${\displaystyle d\mu_{U}=\frac{e^{-U\left(d\right)}d\lambda}{Z}}$
be a probability measure defined in terms of the function $U\left(d\right),$
where $Z$ is the normalization constant.
\begin{thm}
Assume that outside the open unit ball $B=\left\{ d\left(x\right)<1\right\} ,$
the metric $d$ satisfies the following: $\vert\triangledown d\vert$
is bounded, say $\vert\triangledown d\vert\leq1,$ and there exist
finite positive constants $K$ and $c_{0}$ such that
\begin{equation}
\Delta d\leq K+U'\left(d\right)\left(\vert\triangledown d\vert^{2}-c_{0}\right).\label{eq:ref4}
\end{equation}
(i) If $U''\leq\beta U^{'}$ for some positive constant $\beta,$
outside $B$, then for any $q\in\left(1,\infty\right),$ there exist
constants $C_{q},D_{q},$ independent of $f$, such that
\[
\int\vert f\vert^{q}\vert U'\left(d\right)\vert^{q}d\mu_{U}\leq C_{q}\int\vert\triangledown f\vert^{q}d\mu_{U}+D_{q}\int\vert f\vert^{q}d\mu_{U}.
\]
(ii) If, in addition, $U\leq\gamma U'^{q}$ for some positive constant
$\gamma$ and some $q>1,$ outside $B,$ then
\[
\int\vert f\vert^{q}U\left(d\right)d\mu_{U}\leq C_{q}\int\vert\triangledown f\vert^{q}d\mu_{U}+D_{q}\int\vert f\vert^{q}d\mu_{U}.
\]
 In the proof, we adapt, with simplification, the general lines of
proof in {[}6{]}. 
\end{thm}

\begin{proof}
Using integration by parts and condition (\ref{eq:ref4}),
\[
\int\left(\triangledown d\right)\cdot\triangledown\left(fe^{-U\left(d\right)}\right)d\lambda=-\int\Delta d\left(fe^{-U\left(d\right)}\right)d\lambda
\]
\begin{equation}
\geq-K\int fe^{-U\left(d\right)}d\lambda-\int U'\left(d\right)\vert\triangledown d\vert^{2}fe^{-U\left(d\right)}d\lambda+c_{0}\int U'\left(d\right)fe^{-U\left(d\right)}d\lambda.\label{eq:ref5}
\end{equation}
Computing
\[
\triangledown\left(fe^{-U\left(d\right)}\right)=\triangledown fe^{-U\left(d\right)}-U'\left(d\right)e^{-U\left(d\right)}\triangledown d,
\]
and taking the dot product with $\triangledown d,$
\[
\left(\triangledown d\right)\cdot\triangledown\left(fe^{-U\left(d\right)}\right)=\left(\triangledown f\right)\cdot\left(\triangledown d\right)e^{-U\left(d\right)}-U'\left(d\right)f\vert\triangledown d\vert^{2}e^{-U\left(d\right)},
\]
and replacing in (\ref{eq:ref5}) to get
\[
c_{0}\int U'\left(d\right)fe^{-U\left(d\right)}d\lambda\leq K\int fe^{-U\left(d\right)}d\lambda+\int\vert\triangledown f\vert\vert\triangledown d\vert e^{-U\left(d\right)}d\lambda.
\]
Using the fact that $\vert\triangledown d\vert\leq1$ ($\vert\triangledown d\vert$
bounded is enough), we get
\begin{equation}
\int U'\left(d\right)fd\mu_{U}\leq C\int fd\mu_{U}+D\int\vert\triangledown f\vert d\mu_{U}.\label{eq:ref6}
\end{equation}
To prove (i), we replace $f$ in (\ref{eq:ref6}) by $f^{q}$ to get
\[
\int f^{q}\left(U'\left(d\right)\right)^{q}d\mu_{U}=\int\left[f^{q}U'\left(d\right)^{q-1}\right]U'\left(d\right)d\mu_{U}
\]
\[
\leq C\int\left[f^{q}U'\left(d\right)^{q-1}\right]d\mu_{U}+D\int\left|\triangledown\left[f^{q}U'\left(d\right)^{q-1}\right]\right|d\mu_{U}
\]
\[
=C\int ff^{q-1}U'\left(d\right)^{q-1}d\mu_{U}+D\int qf^{q-1}\vert\triangledown f\vert U'\left(d\right)^{q-1}d\mu_{U}
\]
\begin{equation}
+D\int\left(q-1\right)ff^{q-1}U'\left(d\right)^{q-2}U''\left(d\right)\vert\triangledown d\vert d\mu_{U}.\label{eq:ref7}
\end{equation}
Using Young\textquoteright s inequality with $\alpha,$ i.e. ${\displaystyle ab\leq\alpha a^{p}+\frac{b^{q}}{q\left(\alpha p\right)^{q/p}}}$
where $p$ and $q$ are conjugate finite indices in each of the three
terms on the right-hand side of (\ref{eq:ref7}), in addition to using
the conditions $\vert\triangledown d\vert\leq1$ and $U''<\beta U^{'}$
in the last term, we get
\[
\int f^{q}U'\left(d\right)^{q}d\mu_{U}\leq\frac{C}{q\left(\alpha p\right)^{\frac{q}{p}}}\int f^{q}d\mu_{U}+C\alpha\int f^{q}U'\left(d\right)^{q}d\mu_{U}+\frac{D}{\left(\alpha p\right)^{\frac{q}{p}}}\int\vert\triangledown f\vert^{q}d\mu_{U}
\]
\[
+Dq\alpha\int f^{q}U'\left(d\right)^{q}d\mu_{U}+\frac{D\beta\left(q-1\right)}{q\left(\alpha p\right)^{\frac{q}{p}}}\int f^{q}d\mu_{U}+D\left(q-1\right)\alpha\int f^{q}U'\left(d\right)^{q}d\mu_{U}.
\]
Choosing $\alpha$ small enough so that $C\alpha+Dq\alpha+D\left(q-1\right)\alpha\leq1,$
we get 
\[
\int\vert f\vert^{q}U'\left(d\right)^{q}d\mu_{U}\leq C_{q}\int\vert\triangledown f\vert^{q}d\mu_{U}+D_{q}\int\vert f\vert^{q}d\mu_{U}.
\]
\end{proof}
\begin{cor}
If the positive twice differentiable and increasing function $U(d)$
satisfies the two inequalities $U''<\beta U^{'}$ and $U\leq\gamma U'^{q},$
where $\beta$ and $\gamma$ are finite positive constants, then the
measure
\[
d\mu_{U}=\frac{e^{-U\left(d\right)}d\lambda}{Z}
\]

satisfies a q-Poincaré and a q-Logarithmic Sobolev inequality.
\end{cor}

\begin{proof}
To get the Logarithmic Sobolev inequality, we verify inequality (\ref{eq:u2})
from Theorem 7
\[
\vert\triangledown U\left(d\right)\vert^{q}+U\left(d\right)=\vert U'\left(d\right)\triangledown d\vert^{q}+U\left(d\right)\leq\left(1+\gamma\right)U^{'q}\left(d\right)
\]
Since $U''\left(d\right)<\beta U'\left(d\right)$ by assumption, we
can use (i) of Theorem 8, and the result follows. 
\end{proof}
\begin{rem}
In Corollary 9, it is enough to assume that the inequalities
\[
U^{''}<\beta U^{'}\:and\:{}U\leq\gamma U'^{q}
\]
hold almost everywhere for $\beta$ and $\gamma$ finite positive
constants.
\end{rem}

\begin{rem}
To get Theorem 2.3 of \cite{key-58}, take $U\left(d\right)=d^{p},$
then 
\[
U'\left(d\right)=pd^{p-1}\:and\:U''\left(d\right)=p\left(p-1\right)d^{p-2}.
\]

Outside the open unit ball $B=\left\{ d\left(x\right)<1\right\} ,$
\[
\frac{U''\left(d\right)}{U'\left(d\right)}=\frac{p-1}{d}\leq p-1.
\]

Therefore, (i) from Theorem 8 holds, and the statement of Theorem
2.3 of \cite{key-58} is recovered.

Suppose $q$ is the finte conjugate index of $p,$ then,
\[
\frac{U\left(d\right)}{U^{'q}\left(d\right)}=\frac{d^{p}}{p^{q}d^{\left(p-1\right)q}}=\frac{1}{p^{q}}.
\]

Therefore, (ii) of Theorem 8 holds, and the statement of Theorem 2.4
of \cite{key-58} is recovered. 
\end{rem}

\begin{example}
The q-Poincaré and a q-Logarithmic Sobolev inequality are satisfied
for the measure
\[
d\mu_{U}=\frac{e^{-\left(d+1\right)^{p}log\left(d+1\right)}d\lambda}{Z}
\]
for $q\geq\beta$, where $\beta$ is the finite index conjugate to
$p.$ 
\end{example}

\begin{proof}
Given $U\left(d\right)=\left(d+1\right)^{p}log\left(d+1\right).$Then,
\[
U'\left(d\right)=\left(d+1\right)^{p-1}\left(plog\left(d+1\right)+1\right),
\]
and
\[
U''\left(d\right)=\left(d+1\right)^{p-2}\left(p\left(p-1\right)log\left(d+1\right)+2p-1\right).
\]
So, outside the open unit ball $B=\left\{ d\left(x\right)<1\right\} ,$
\[
\frac{U''\left(d\right)}{U'\left(d\right)}=\frac{\left(p\left(p-1\right)log\left(d+1\right)+2p-1\right)}{d+1}\leq p\left(p-1\right)+\frac{2p-1}{d+1}\leq p^{2}-\frac{1}{2}.
\]
In addition,
\[
\frac{U\left(d\right)}{U^{'q}\left(d\right)}=\frac{\left(d+1\right)^{p}log\left(d+1\right)}{\left(d+1\right)^{\left(p-1\right)q}\left(plog\left(d+1\right)+1\right)^{q}}\leq\frac{\left(d+1\right)^{p}log\left(d+1\right)}{\left(d+1\right)^{\left(p-1\right)\beta}\left(plog\left(d+1\right)+1\right)^{q}}\leq1.
\]
So, by Corollary 9, the q-Poincaré and a q-Logarithmic Sobolev inequality
are satisfied for the measure ${\displaystyle d\mu_{U}=\frac{e^{-\left(d+1\right)^{p}log\left(d+1\right)}}{Z}d\lambda}$
for $q\geq\beta,$ where $\beta$ is the finite index conjugate to
$p.$ 
\end{proof}
\begin{example}
For $U\left(d\right)=sinh\left(d\right),$ $U\left(d\right)=U^{''}\left(d\right)\leq cosh\left(d\right)=U^{'}\left(d\right).$

So, by Corollary 9, the q-Poincaré and q-Logarithmic Sobolev inequalities
hold true for the measure ${\displaystyle d\mu_{U}=\frac{e^{-sinh\left(d\right)}}{Z}d\lambda}$
for all $q\geq1.$
\end{example}

\section{$p-$Talagrand inequality for ${\displaystyle d\mu=\frac{e^{-U(d)}}{Z}d\lambda}$}

To prove that the $q-$Logarithmic Sobolev inequality implies the
$p-$Talagrand inequality, we are inspired by the proof of Theorem
4.1 in \cite{key-67}. However, instead of using the Hamilton-Jacobi
equation as in that paper, we use the following Hamilton-Jacobi equation
as in Theorem 4 in the paper by F. Dragoni \cite{key-70}. In the
setting of Carnot groups, the restriction for using a doubling measure
$\mu$ is no longer needed.

Let $d$ be the Carnot-Carathéodory distance and ${\displaystyle d\mu=\frac{e^{-U(d)}}{Z}d\lambda}$
be a measure where $Z$ is a normalization constant, $\lambda$ is
the Lebesgue measure, and $U(d)$ satisfies the conditions of Theorem
8.
\begin{thm}
Let $1<q\leq2,$ and $p\geq2$ be its finite index conjugate, so that
${\displaystyle \frac{1}{p}+\frac{1}{q}=1.}$

If $(\mathbb{G},d,\mu)$ satisfies the q-Logarithmic Sobolev \ref{eq:lsq}
inequality with constant $c={\displaystyle (q-1)\left(\frac{q}{K}\right)^{q-1}}$
for some constant $K>0,$ then it also satisfies the p-Talagrand inequality
with the same constant $K$.
\end{thm}

For every continuous bounded function $f$ where $Qf=Q_{1}f,$ we
will be proving the dual formulation of the Talagrand inequality
\begin{equation}
\int_{\mathbb{G}}e^{KQf}d\mu\leq e^{K\int_{\mathbb{G}}fd\mu},\label{eq:dt}
\end{equation}

which is shown equivalent to the $p-$Talagrand inequality (\ref{eq:4})
in \cite{key-54}.
\begin{proof}
For some $n\geq1,$ define 
\[
\phi(t)=\frac{1}{Kt^{n}}log\left(\int_{\mathbb{G}}e^{kt^{n}Q_{t}f}d\mu\right).
\]
Using the fact that $f$ is bounded and the dominated convergence
theorem, Balogh et al. prove in Theorem 4.1 of \cite{key-67} that:
\[
\underset{t\rightarrow0^{+}}{lim}\phi(t)=\int_{\mathbb{G}}fd\mu.
\]
The main goal is to prove that $\phi(t)$ is non-increasing in $t$,
which implies that $\phi(1)\leq\underset{t\rightarrow0^{+}}{lim}\phi(t).$
In other words, replacing:
\[
\phi(1)=\frac{1}{K}log\left(\int_{\mathbb{G}}e^{kQf}d\mu\right)\leq\underset{t\rightarrow0^{+}}{lim}\phi(t)=\int_{\mathbb{G}}fd\mu.
\]
Rearranging the terms and exponentiating the last inequality, we get
the dual Talagrand inequality (\ref{eq:dt}). We now show that $\phi(t)$
is non-increasing. The two main ingredients to prove that are the
q-Logarithmic Sobolev inequality and the solution to the Hamilton-Jacobi
problem (\ref{eq:}).

Fix $t\in(0,1].$ For $s>0,$ we have 

\[
\underset{s\rightarrow0^{+}}{lim}\frac{\phi(t+s)-\phi(t)}{s}=\underset{s\rightarrow0^{+}}{lim}\frac{1}{s}\left(\frac{1}{K(t+s)^{n}}-\frac{1}{Kt^{n}}\right)log\int_{\mathbb{G}}e^{K(t+s)^{n}Q_{t+s}f}d\mu
\]

\begin{equation}
+\underset{s\rightarrow0^{+}}{lim}\frac{1}{Kt^{n}s}\left(log\int_{\mathbb{G}}e^{K(t+s)^{n}Q_{t+s}f}d\mu-log\int_{\mathbb{G}}e^{Kt^{n}Q_{t}f}d\mu\right).\label{eq:2-1}
\end{equation}

As $s\rightarrow0^{+},$ the first term on the right-hand side of
(\ref{eq:2-1}) converges to 

\[
-\frac{n}{Kt^{n+1}}log\left(\int_{\mathbb{G}}e^{Kt^{n}Q_{t}f}d\mu\right).
\]

The limit of the second term of (\ref{eq:2-1}) is

\[
\frac{1}{Kt^{n}}\frac{1}{\int_{\mathbb{G}}e^{Kt^{n}Q_{t}f}d\mu}\underset{s\rightarrow0^{+}}{lim}\left[\frac{1}{s}\left(\int_{\mathbb{G}}e^{K(t+s)^{n}Q_{t+s}f}d\mu-\int_{\mathbb{G}}e^{Kt^{n}Q_{t}f}d\mu\right)\right]
\]

rewriting, we get

\[
=\frac{1}{Kt^{n}}\frac{1}{\int_{\mathbb{G}}e^{Kt^{n}Q_{t}f}d\mu}\underset{s\rightarrow0^{+}}{lim}\left[\int_{\mathbb{G}}\frac{e^{K(t+s)^{n}Q_{t+s}f}-e^{Kt^{n}Q_{t+s}f}}{s}d\mu+\int_{\mathbb{G}}\frac{e^{Kt^{n}Q_{t+s}f}-e^{Kt^{n}Q_{t}f}}{s}d\mu\right]
\]

so, using the dominated convergence theorem for the first term, we
get

\begin{equation}
=\frac{1}{Kt^{n}}\frac{\int_{\mathbb{G}}Knt^{n-1}Q_{t}fe^{Kt^{n}Q_{t}f}d\mu}{\int_{\mathbb{G}}e^{Kt^{n}Q_{t}f}d\mu}+\frac{1}{Kt^{n}}\frac{1}{\int_{\mathbb{G}}e^{Kt^{n}Q_{t}f}d\mu}\underset{s\rightarrow0^{+}}{lim}\left[\int_{\mathbb{G}}\frac{e^{Kt^{n}Q_{t+s}f}-e^{Kt^{n}Q_{t}f}}{s}d\mu\right].\label{eq:3-1}
\end{equation}

In order to compute the limit in (\ref{eq:3-1}), we use the Hamilton-Jacobi
equation in Theorem 5, which has the solution $Q_{t}f,$ so
\[
\frac{\partial Q_{t}f}{\partial_{t}}=-\frac{\left|\triangledown Q_{t}f\right|^{q}}{q},
\]

so,
\begin{equation}
\underset{s\rightarrow0^{+}}{lim}\left(\frac{e^{Kt^{n}Q_{t+s}f}-e^{Kt^{n}Q_{t}f}}{s}\right)=Kt^{n}e^{Kt^{n}Q_{t}f}\frac{\partial Q_{t}f}{\partial_{t}}=-Kt^{n}e^{Kt^{n}Q_{t}f}\frac{\left|\triangledown Q_{t}f\right|^{q}}{q}.\label{eq:4-1}
\end{equation}

Since $Q_{(.)}f(.)$ is Lipschitz on $\mathbb{G}\times\mathbb{R}_{+},$ $Q_{t+s}f=Q_{t}f+O(s)$
holds uniformly on $\mathbb{G}.$ Since $Q_{t}f(x)$ is uniformly bounded in
$x,$ we can use the dominated convergence theorem and (\ref{eq:4-1}),
to find the limit in (\ref{eq:3-1}):
\[
\underset{s\rightarrow0^{+}}{lim}\left[\int_{\mathbb{G}}\frac{e^{Kt^{n}Q_{t+s}f}-e^{Kt^{n}Q_{t}f}}{s}d\mu\right]=-Kt^{n}\int_{\mathbb{G}}e^{Kt^{n}Q_{t}f}\frac{\left|\triangledown Q_{t}f\right|^{q}}{q}d\mu.
\]
In other words, equation (\ref{eq:3-1}) becomes: 
\[
\underset{s\rightarrow0^{+}}{lim}\frac{1}{Kt^{n}s}\left(log\int_{\mathbb{G}}e^{K(t+s)^{n}Q_{t+s}f}d\mu-log\int_{\mathbb{G}}e^{Kt^{n}Q_{t}f}d\mu\right)
\]
\[
=\frac{1}{Kt^{n+1}}\frac{1}{\int_{\mathbb{G}}e^{Kt^{n}Q_{t}f}d\mu}\left[\int_{\mathbb{G}}nKt^{n}Q_{t}fe^{Kt^{n}Q_{t}f}d\mu-\int_{\mathbb{G}}Kt^{n+1}\frac{\left|\triangledown Q_{t}f\right|^{q}}{q}e^{Kt^{n}Q_{t}f}d\mu\right].
\]
Replacing the last equation in (\ref{eq:2-1}), we get:
\[
\underset{s\rightarrow0^{+}}{lim}\frac{\phi(t+s)-\phi(t)}{s}=\frac{1}{Kt^{n+1}}\frac{1}{\int_{\mathbb{G}}e^{Kt^{n}Q_{t}f}d\mu}
\]
\[
\times[-nlog\left(\int_{\mathbb{G}}e^{Kt^{n}Q_{t}f}d\mu\right)\int_{\mathbb{G}}e^{Kt^{n}Q_{t}f}d\mu+\int_{\mathbb{G}}nKt^{n}Q_{t}fe^{Kt^{n}Q_{t}f}d\mu
\]
\[
-\int_{\mathbb{G}}Kt^{n+1}\frac{\left|\triangledown_{\mathbb{}}Q_{t}f\right|^{q}}{q}e^{Kt^{n}Q_{t}f}d\mu].
\]
Let ${\displaystyle g=e^{\frac{1}{q}Kt^{n}Q_{t}f}},$ and ${\displaystyle n=\frac{1}{q-1},}$
we obtain:
\[
\underset{s\rightarrow0^{+}}{lim}\frac{\phi(t+s)-\phi(t)}{s}=\frac{1}{Kt^{\frac{q}{q-1}}}\frac{1}{\int_{\mathbb{G}}g^{q}d\mu}
\]

\[
\times[-\frac{1}{q-1}log\left(\int_{\mathbb{G}}g^{q}d\mu\right)\int_{\mathbb{G}}g^{q}d\mu+\frac{1}{q-1}\int_{\mathbb{G}}g^{q}log(g^{q})d\mu
\]
\begin{equation}
-\int_{\mathbb{G}}Kt^{\frac{q}{q-1}}\frac{\left|\triangledown Q_{t}f\right|^{q}}{q}g^{q}d\mu].\label{eq:5-1}
\end{equation}

Using the q-Logarithmic Sobolev inquality 
\[
Ent_{\mu}(g^{q})=\int_{\mathbb{G}}g^{q}log(g^{q})d\mu-log\left(\int_{\mathbb{G}}g^{q}d\mu\right)\int_{\mathbb{G}}g^{q}d\mu
\]
\[
\leq(q-1)\left(\frac{q}{K}\right)^{q-1}\int_{\mathbb{G}}|\triangledown g|^{q}d\mu
\]

in (\ref{eq:5-1}):
\[
\underset{s\rightarrow0^{+}}{lim}\frac{\phi(t+s)-\phi(t)}{s} \leq \frac{1}{Kt^{\frac{q}{q-1}}}\frac{1}{\int_{\mathbb{G}}g^{q}d\mu}
\]

\[
\times\left[\left(\frac{q}{K}\right)^{q-1}\int_{\mathbb{G}}|\triangledown g|^{q}d\mu-\int_{\mathbb{G}}Kt^{\frac{q}{q-1}}\frac{\left|\triangledown_{\mathbb{}}Q_{t}f\right|^{q}}{q}g^{q}d\mu\right]
\]
\[
=\frac{1}{Kt^{\frac{q}{q-1}}}\frac{1}{\int_{\mathbb{G}}g^{q}d\mu}\left[\left(\frac{q}{K}\right)^{q-1}\int_{\mathbb{G}}|\triangledown{\displaystyle (e^{\frac{1}{q}Kt^{n}Q_{t}f})}|^{q}d\mu-\int_{\mathbb{G}}Kt^{\frac{q}{q-1}}\frac{\left|\triangledown Q_{t}f\right|^{q}}{q}g^{q}d\mu\right]=0.
\]

Hence, $\phi(t)$ is non-increasing, and we obtain the dual Talagrand
inequality . 
\end{proof}

\section{ Hypercontractivty of Hamilton-Jacobi solutions for ${\displaystyle d\mu=\frac{e^{-U(d)}}{Z}d\lambda}$}

Denote $||.||_{p},$ $p\in\mathbb{R},$ the $L^{p}-$norms with respect
to $\mu,$ and ${\displaystyle ||f||_{0}=e^{\int log|f|d\mu}}$ whenever
$log|f|$ is $\mu-$integrable. Also, consider the Hamilton-Jacobi
problem on a Carnot group $\mathbb{G}$ as in Theorem 5. The following
Hypercontractivity theorem is inspired by Theorem 2.1 by Bobkov et
al. in \cite{key-53}, but with the measure ${\displaystyle d\mu=\frac{e^{-U(d)}}{Z}d\lambda}$
and the Euclidean gradient replaced by the sub-gradient $\triangledown.$
\begin{thm}
Assume we have the following 2-Logarithmic Sobolev inequality with
the measure ${\displaystyle d\mu=\frac{e^{-U(d)}}{Z}d\lambda}$, and
in the setting of the Carnot group:

\[
\rho Ent_{\mu}(f^{2})=\rho\int_{\mathbb{G}}f^{2}log(f^{2})d\mu-\rho log\left(\int_{\mathbb{G}}f^{2}d\mu\right)\int_{\mathbb{G}}f^{2}d\mu\leq2\int_{\mathbb{G}}|\triangledown f|^{2}d\mu.
\]

Then, for every bounded measurable function $f$ on $\mathbb{G},$ every $t\geq0,$
and every $a\in\mathbb{R},$

\[
||e^{Q_{t}f}||_{a+\rho t}\leq||e^{f}||_{a}.
\]
\end{thm}

The following proof adapts the proof of Theorem 2.1 done by Bobkov
et al. in \cite{key-53} in the Euclidean setting, but uses the Hamilton-Jacobi
problem in the Carnot group setting (Theorem 5).
\begin{proof}
Let ${\displaystyle F(t)=||e^{Q_{t}f}||_{\lambda(t)},}$ with $\lambda(t)=a+\rho t,$
$t>0.$

For all $t>0$ and almost every $x,$ the partial derivatives ${\displaystyle \frac{\partial}{\partial_{t}}Q_{t}f(x)}$
exist, and by the Hamilton-Jacobi problem (Theorem 5),

\begin{equation}
{\displaystyle \frac{\partial}{\partial_{t}}Q_{t}f(x)}=-\frac{1}{2}|\triangledown Q_{t}f(x)|^{2}.\label{eq:6}
\end{equation}

So, $F$ is differentiable at every point $t>0$ where $\lambda(t)\neq0.$ 

Computing,
\begin{equation}
\left(F(t)^{\lambda(t)}\right)'=\rho F(t)^{\lambda(t)}log(F(t))+\lambda(t)F(t)^{\lambda(t)-1}F'(t).\label{eq:7}
\end{equation}
Also,
\[
\left(F(t)^{\lambda(t)}\right)'=\frac{\partial}{\partial_{t}}\left(\int_{\mathbb{G}}e^{\lambda(t)Q_{t}f}d\mu\right)
\]
\[
=\int_{\mathbb{G}}\left(\rho Q_{t}f+\lambda(t)\frac{\partial}{\partial_{t}}Q_{t}f\right)e^{\lambda(t)Q_{t}f}d\mu
\]
using (\ref{eq:6}),
\[
=\frac{\rho}{\lambda(t)}Ent_{\mu}(e^{\lambda(t)Q_{t}f})+\frac{\rho}{\lambda(t)}log\left(\int_{\mathbb{G}}e^{\lambda(t)Q_{t}f}d\mu\right)\int_{\mathbb{G}}e^{\lambda(t)Q_{t}f}d\mu-\frac{\lambda(t)}{2}\int_{\mathbb{G}}|\triangledown Q_{t}f(x)|^{2}e^{\lambda(t)Q_{t}f}d\mu
\]
\[
=\frac{\rho}{\lambda(t)}Ent_{\mu}(e^{\lambda(t)Q_{t}f})+\rho F(t)^{\lambda(t)}log\left(F(t)\right)-\frac{\lambda(t)}{2}\int_{\mathbb{G}}|\triangledown Q_{t}f(x)|^{2}e^{\lambda(t)Q_{t}f}d\mu.
\]
Combining this with (\ref{eq:7}), we get
\[
\lambda(t)^{2}F(t)^{\lambda(t)-1}F'(t)=\rho Ent_{\mu}(e^{\frac{1}{2}\lambda(t)Q_{t}f})^{2}-\frac{\lambda(t)^{2}}{2}\int_{\mathbb{G}}|\triangledown Q_{t}f(x)|^{2}e^{\lambda(t)Q_{t}f}d\mu
\]
using the 2-Logarithmic Sobolev inequality,
\[
\leq\int_{\mathbb{G}}|\triangledown_{\mathbb{G}}e^{\frac{\lambda(t)Q_{t}f}{2}}|^{2}d\mu-\frac{\lambda(t)^{2}}{2}\int_{\mathbb{G}}|\triangledown Q_{t}f(x)|^{2}e^{\lambda(t)Q_{t}f}d\mu
\]
\[
=\frac{\lambda(t)^{2}}{4}\int_{\mathbb{G}}|\triangledown_{\mathbb{G}}Q_{t}f(x)|^{2}e^{\lambda(t)Q_{t}f}d\mu-\frac{\lambda(t)^{2}}{2}\int_{\mathbb{G}}|\triangledown Q_{t}f(x)|^{2}e^{\lambda(t)Q_{t}f}d\mu
\]
\[
\leq0,
\]
for all $t>0.$ Thus, $F'(t)\leq0$ for all $t>0.$ Since $F$ is
continuous, then it is non-increasing. Hence, $F(t)\leq F(0),$ for
all $t>0.$
\[
||e^{Q_{t}f}||_{\lambda(t)}\leq||e^{Q_{0}f}||_{\lambda(0)}.
\]
So,
\[
||e^{Q_{t}f}||_{a+\rho t}\leq||e^{f}||_{a}.
\]
\end{proof}

\end{document}